\documentclass[12pt]{amsart}
\usepackage{amsfonts}
\usepackage{amsmath}
\usepackage{amssymb}
\usepackage{a4}
\usepackage{multirow}
\usepackage{hyperref}

\newcommand{\heute}{20 August 2010}

\theoremstyle{plain}
\newtheorem{theorem}{Theorem}[section]

\newtheorem{lemma}[theorem]{Lemma}
\newtheorem{corollary}[theorem]{Corollary}
\newtheorem{proposition}[theorem]{Proposition}
\theoremstyle{remark}
\newtheorem{conj}[theorem]{Conjecture}
\newtheorem{remark}[theorem]{Remark}
\newtheorem{remarks}[theorem]{Remarks}

\newtheorem{ques}[theorem]{Question}
\newtheorem*{defn}{Definition}
\newtheorem*{rk}{Remark}
\newtheorem*{bsp}{Example}
\newtheorem*{bspe}{Examples}

\newcommand{\dashTwo}[1]{\textup{(\ref{two}${}'$)}}

\newcommand{\ignore}[1]{}

\newcommand{\zz}{\mathbb{Z}}

\newcommand{\depth}{\operatorname{depth}}

\newcommand{\Ann}{\operatorname{Ann}}
\newcommand{\Ass}{\operatorname{Ass}}

\newcommand{\CMd}{\delta}
\newcommand{\gtD}{\delta_0}
\newcommand{\De}{e}
\newcommand{\Aplus}{A_+}

\DeclareMathSymbol\normal{\mathrel}{AMSa}{"43}

\newcommand{\prank}[1][p]{\text{$#1$-rk}}

\newcommand{\abs}[1]{\left|#1\right|}

\begin{document}

\title[Cohomology of groups of order 128]{The computation
of the cohomology rings of all groups of order 128}
\author[D. J. Green]{David J. Green}
\address{Mathematical Institute \\
University of Jena \\ D-07737 Jena \\
Germany}
\email{david.green@uni-jena.de}
\author[S. A. King]{Simon A. King}
\address{Mathematics Department \\ National University of Ireland \\ Galway \\
Ireland}
\email{simon.king@nuigalway.ie}
\thanks{King was supported by DFG grant GR 1585/4-1 during most of this research.
Both authors received travel assistance from this grant and from GR 1585/4-2\@.}
\date{\heute}

\begin{abstract}
\noindent
We describe the computation of the mod-2 cohomology rings of all
2328 groups of order 128\@. One consequence is that all groups of order
less than 256 satisfy the strong form of Benson's Regularity Conjecture.
\end{abstract}

\maketitle
\newlength{\djglength}

\section{Introduction}
\noindent
Computing large numbers of group cohomology rings allows one to test existing conjectures and
to look for new patterns. Carlson's computations~\cite[Appendix]{CarlsonTownsley} for all 267 groups of order 64
led to the refutation of the Essential Conjecture~\cite[p.~96]{habil}, and inspired further work on
the significance of the local cohomology of group cohomology rings: see pp.~6--8 of~\cite{Benson:MSRI},
and its appendix.

Here we announce the computation of the mod-2 cohomology rings of all 2328 groups of
order 128\@.
The results are available in human-readable form at our
website~\cite{GreenKing:128website}\@. They can be read into the
computer using our package~\cite{SimonsProg}
for the Sage~\cite{Sage} computer algebra system.

We also give some first applications.
Theorem~\ref{thm:main} verifies the
strong form of Benson's Regularity
Conjecture for all groups of order less than 256\@.
We study the statistical effectiveness of Duflot's
lower bound for the depth in Section~\ref{sect:Duflot}\@.
And we investigate the local cohomology of these cohomology rings, observing that
although the $a$-invariants are normally weakly increasing, there are 15
exceptions amongst the 2595 groups of order 64 or 128:
see Table~\ref{table:nonMonotone} on page~\pageref{table:nonMonotone}\@.
\medskip

\noindent
Following J.~F. Carlson~\cite{Carlson:Tests},
we compute mod-$p$ cohomology rings of $p$-groups from
a suitably large initial segment of the minimal projective
resolution, using a completeness criterion to tell when we are done.
For a project of this scale it was necessary to detect completion as soon as possible, as
each extra stage in the minimal resolution is costly.
We derive in Theorem~\ref{thm:BensonVariant} a computationally efficient variant
of Benson's test for completion~\cite[Theorem~10.1]{Benson:DicksonCompCoho}\@.
Benson's test involves finding filter-regular systems of parameters. In
Sections \ref{sect:freg}~and \ref{sect:existence} we develop methods for constructing such
parameters in reasonably low degrees, and demonstrating their existence in lower degrees.
\medskip

\noindent
We implemented our cohomology computations in the open source computer algebra system Sage~\cite{Sage}\@.
King rewrote large parts of
Green's cohomology program~\cite{habil}\@. He developed extensive Cython code that implements the theoretical improvements presented in this paper and makes use of various components of Sage, most notably \textsc{Singular}~\cite{Singular}\@.
%The new package makes extensive use of the computer algebra system
%\textsc{Singular}~\cite{Singular}\@, and in order to make the various
%components of the program work well together we used
%Sage~\cite{Sage}\@, developing extensive Cython code. 
Our Sage package is available online~\cite{SimonsProg}\@ and
contains the cohomology ring for each group of order 64. 
The cohomology ring of each group of order 128 is automatically
downloaded from a web repository when it is required. The cohomology
rings and induced homomorphisms can then easily be transformed into
objects of \textsc{Singular}\@.

\section{Constructing filter-regular parameter systems}
\label{sect:freg}
\noindent
We recall the key concepts from Benson's paper~\cite{Benson:DicksonCompCoho}\@.
Let $k$ be a field of characteristic $p>0$. Thorought this paper, let 
$A = \bigoplus_{n\geq 0} A_n$ be a connected noetherian graded commutative $k$-algebra.
That is, $A_0=k$ and $A$ is finitely generated.

Elements of~$A$ will always be understood to be homogeneous. If $x \in
A_n$ then we write $\abs{x} = n$. Recall that the elements of positive
degree form a maximal ideal, which we shall denote by~$\Aplus$.

\begin{defn}
\emph{\cite[\S3]{Benson:DicksonCompCoho}}
\begin{enumerate}
\item
A sequence of elements $h_1,\ldots,h_r \in A_+$ is called
\emph{filter-regular} if for each $1 \leq i \leq r$ the annihilator
of~$h_i$ in the quotient ring $A/(h_1,\ldots,h_{i-1})$ has bounded
degree (\emph{i.e.}, it is a finite dimensional $k$-vector space).
% \item
% A filter-regular system of parameters is a system of parameters which is
% simultaneously a filter-regular sequence. Note that in this case
% the quotient $A/(h_1,\ldots,h_r)$ also has bounded degree.
\item
A system of parameters $h_1,\ldots,h_r$ which is filter-regular is said to have
\emph{type} $(d_0,d_1,\ldots,d_r) \in \zz^{r+1}$ if
\begin{xalignat*}{2}
d_0 & \geq -1 \, , & d_{i-1} - 1 \leq d_i & \leq d_{i-1} \; \text{for each $1 \leq i \leq r$,}
\end{xalignat*}
and the annihilator of $h_{i+1}$ in $A/(h_1,\ldots,h_i)$ lies in
degrees${} \leq d_i + \sum_{j=0}^i \abs{h_i}$ for all $0 \leq i \leq
r$, where $h_{r+1} = 0$.
\item
A system of parameters is \emph{strongly quasi-regular} if it is filter-regular
of type $(0,-1,\ldots,-r)$.
If it has type $(-1,-2,\ldots,-r,-r)$ then it is
\emph{very strongly quasi-regular}\@.
.
\end{enumerate}
\end{defn}

\begin{lemma}
\label{lemma:fregBasics}
Let $h_1,\ldots,h_r$ be a sequence of homogeneous elements in~$\Aplus$.
Let $d$ be the Krull dimension of~$A$.
Then the following are equivalent:
\begin{enumerate}
\item
$h_1,\ldots,h_r$ is a filter-regular system of parameters.
\item
$h_1,\ldots,h_r$ is a filter-regular sequence, $r \leq d$,
and $A/(h_1,\ldots,h_r)$ is a finite dimensional $k$-vector space.
\item
$h_1,\ldots,h_r,0$ is a filter-regular sequence and $r \leq d$.
\end{enumerate}
\end{lemma}

\begin{rk}
If $A$ has bounded degree, then \emph{every} sequence in~$A$ is filter-regular.
\end{rk}

\begin{proof}
Recalling what filter-regularity means for the term $h_{r+1} = 0$, one sees
the equivalence of the last two statements. The first implies the second,
since then $r=d$ and $A/(h_1,\ldots,h_r)$ has Krull dimension zero.
And the second implies the first, for $A$ is a finite module for
$k[h_1,\ldots,h_r]$, and so if $h_1,\ldots,h_r$ were algebraically
dependent then the Krull dimension would have to be less than~$r \leq d$.
\end{proof}

\begin{rk}
By Benson's \cite[Theorem~4.5]{Benson:DicksonCompCoho} all filter-regular parameter systems
for~$A$ have the same type.
By his Corollary~4.8 and Symonds' Theorem~\cite{Symonds:RegProof},
every cohomology ring $H^*(G,k)$ has a strongly quasi-regular parameter system.
\end{rk}

\subsection*{The weak rank-restriction condition}
Benson's test requires an explicit
filter-regular system of parameters in~$H^*(G,k)$.
One approach is to find
classes whose restrictions to each elementary abelian subgroup are
(powers of) the Dickson invariants~\cite[Corollary~9.8]{Benson:DicksonCompCoho}\@.
This is straightforward, but the degrees involved
can be inconveniently large.
Restricting to the case of a $p$-group, we shall see in
Lemma~\ref{lemma:filterConstruct} that one can use
the Dickson invariants for a complement of the centre instead: these lie in lower degree.

\begin{defn}
(c.f.\@ \cite[\S8]{Benson:DicksonCompCoho}) \quad
Let $G$ be a $p$-group with $\prank(G)=K$. Set $C$ to be
$\Omega_1(Z(G))$, the greatest central elementary abelian subgroup of~$G$\@.
Homogeneous elements $\zeta_1,\ldots,\zeta_K
\in H^*(G)$ satisfy the \emph{weak rank restriction condition} if
for each elementary abelian subgroup $V \geq C$ of $G$ the
following holds, where $s = \prank(V)$:
\begin{quote}
\noindent
The restrictions of $\zeta_1,\ldots,\zeta_s$ to~$V$ form a
homogeneous system of parameters for $H^*(V)$; and
the restrictions of $\zeta_{s+1},\ldots,\zeta_K$ are zero.
\end{quote}
\end{defn}

\begin{lemma}
If $\zeta_1,\ldots,\zeta_K \in H^*(G)$ satisfy the weak rank restriction
condition then they constitute a filter-regular system of parameters.
\end{lemma}

\begin{proof}
By a well known theorem of Quillen (see e.g.\@ \cite[\S5.6]{Benson:II}),
$\zeta_1,\ldots,\zeta_K$ is a system of parameters for $H^*(G)$.
The proof of~\cite[Theorem~9.6]{Benson:DicksonCompCoho} applies
equally well to parameters satisfying the weak rank restriction condition:
if $E \leq G$ is elementary abelian of rank~$i$,
then for $V = \langle C, E\rangle$ one has $V \geq C$ and
$C_G(V)=C_G(E)$. And since $\prank(V) \geq i$,
the restrictions of $\zeta_1,\ldots,\zeta_i$ to $H^*(C_G(E))$ form a
regular sequence, by the same Duflot-type argument.
\end{proof}

\begin{lemma}
\label{lemma:filterConstruct}
Let $G$ be a $p$-group. Set with $K = \prank(G)$, $C = \Omega_1(Z(G))$ and $r = \prank(C)$.
Then there exist
$\zeta_1,\ldots,\zeta_K \in H^*(G)$ satisfy the following conditions:
\begin{enumerate}
\item
The restrictions of $\zeta_1,\ldots,\zeta_r$
form a regular sequence in $H^*(C)$.
\item
For each rank $r+s$ elementary abelian subgroup $V \geq C$ of $G$ the
restrictions of $\zeta_{r+s+1},\ldots,\zeta_K$ to $V$ are zero,
and for $1 \leq i \leq s$ the restrictions of $\zeta_{r+i}$ to~$V$ is
a power of the $i$th Dickson invariant in $H^*(V/C)$.
\end{enumerate}
Any such system $\zeta_1,\ldots,\zeta_K$ is a filter-regular system of
parameters for $H^*(G)$.
\end{lemma}

\begin{rk}
By the $i$th Dickson invariant we mean the one which restricts nontrivially
to dimension~$i$ subspaces, but trivially to smaller subspaces.
That is, if $i < j$ then the $i$th Dickson invariant lies in lower degree
than the $j$th Dickson invariant.
\end{rk}

\begin{proof}
Evens showed \cite[p.~128]{CarlsonTownsley} that $H^*(C)$ is a finitely
generated module over the image of restriction, so
$\zeta_1,\ldots,\zeta_r$ exist.
The~$\zeta_{r+i}$ are specified by their
restrictions to elementary abelian subgroups. These restrictions
satisfy the compatibility conditions for genuine
restrictions. Thus, on raising these defining
restrictions by sufficiently high $p$th powers, the $\zeta_{r+i}$
do indeed exist, by a result of Quillen (\cite[Coroll 5.6.4]{Benson:II}).
Last part: The weak rank restriction condition is satisfied.
\end{proof}

\begin{rk}
Lemma~\ref{lemma:filterConstruct} is a recipe for
constructing filter-regular parameters.
By Kuhn's work~\cite{Kuhn:Cess},
$\zeta_1,\ldots,\zeta_r$ may be found
amongst the generators of~$H^*(G)$. 

Moreover one can replace $\zeta_K$ by
any element that completes a system of parameters, as every parameter is filter-regular
in the one-dimensional case.
\end{rk}

\begin{bsp}
The Sylow $2$-subgroup of $\mathit{Co}_3$ has order $1024$.
The Dickson invariants lie in degrees 8, 12, 14 and 15\@, so they can only be used
for Benson's test in degree $46$\@.
Lemma~\ref{lemma:filterConstruct} yields filter regular parameters in degrees 8, 4, 6 and 7,
and the last parameter can be replaced by one in degree~$1$. 
This allows one to use Benson's test in degree $17$\@. In the next section we will
improve this to~$14$\@.
\end{bsp}

\subsection*{New filter-regular seqences from old}
Finding filter-regular parameters in low degrees makes the computations easier.
The two following methods can be used to lower the degrees once filter-regular parameters
have been found. They are factorization (Lemma~\ref{lemma:factor}) and
nilpotent alteration (Lemma~\ref{lemma:fregNilp})\@.

\begin{lemma}
\label{lemma:fakAux}
Suppose that $p_1,\ldots,p_r \in A_+$ are
filter-regular. Suppose that $p_1 x \in (p_2,\ldots,p_r)$ and that
$x$~lies in sufficiently high degree. Then $x \in (p_2,\ldots,p_r)$.
\end{lemma}

\begin{proof}
This is clear for $r=1$. We proceed by induction on~$r$.  So suppose that
$r \geq 2$ and
$p_1 x = \sum_{i=2}^r p_i w_i$.

Since $w_r$ is in sufficiently large degree, there are
$v_1,\ldots,v_{r-1} \in A$ with
$w_r = \sum_{i=1}^{r-1} p_i v_i$. Then
$p_1(x - p_r v_1) \in (p_2,\ldots,p_{r-1})$, and we are done by induction.
\end{proof}

\begin{lemma}
\label{lemma:factor}
% Let $A$ be a connected noetherian graded commutative $k$-algebra.
Suppose that $p_1,\ldots,p_r \in A_+$ is a filter-regular sequence, and
that $f_1,\ldots,f_r \in A_+$ satisfy $f_i \mid p_i$ for each~$i$.
Then $f_1,\ldots,f_r$ is filter-regular too.
Hence if $p_1,\ldots,p_r$ is a filter-regular system of parameters, then so is
$f_1,\ldots,f_r$.
\end{lemma}

\begin{proof}
By induction on $r$ it suffices to prove that $f_1,p_2,\ldots,p_r$ is filter-regular.
We do this by induction on~$r$.
If $f_1 = p_1$ then we are done, so we may assume that $p_1 = f_1 g_1$ with
$f_1,g_1 \in A_+$.
If $f_1 x = 0$ then $p_1 x = 0$, so $x$~is in bounded degree
and the case $r=1$ is done. Now suppose that $r \geq 1$. Observe that we may
assume that both $f_1,p_2,\ldots,p_{r-1}$ and $g_1,p_2,\ldots,p_{r-1}$ are filter-regular.

Suppose that $p_rx \in (f_1,p_2,\ldots,p_{r-1})$ and that $x$~is in large degree. Then
$p_r g_1 x \in (p_1,p_2,\ldots,p_{r-1})$ and $g_1 x$ is in large degree, so
by filter-regularity of $p_1,\ldots,p_r$ we have
$g_1 x \in (p_1,\ldots,p_{r-1})$.

Hence for some $w \in A$ we have
$g_1(x - f_1 w) \in (p_2,\ldots,p_{r-1})$.
So $x - f_1 w \in (p_2,\ldots,p_{r-1})$ by Lemma~\ref{lemma:fakAux},
as $g_1,p_2,\ldots,p_{r-1}$ is filter regular.
Therefore $x \in (f_1,p_2,\ldots,p_{r-1})$, as required.
Last part: Follows from Lemma~\ref{lemma:fregBasics}\@.
\end{proof}

\begin{lemma}
\label{lemma:fregPowers}
Suppose that $p_1,\ldots,p_r \in A_+$,
and that $m_1,\ldots,m_r \geq 1$. Then
$p_1,\ldots,p_r$ is filter-regular
if and only if $p_1^{m_1},\ldots,p_r^{m_r}$ is.
\end{lemma}

\begin{proof}
One direction follows by Lemma~\ref{lemma:factor}\@.
For the other it suffices to prove that if
$p_1,\ldots,p_r$ is filter-regular, then 
$p_1^m,p_2,\ldots,p_r$ is too.
This is true for $m=1$, so assume that $m \geq 2$ and that it is true for
$m-1$.  Suppose that $x$ has sufficiently large
degree and satisfies $p_r x \in (p_1^m,p_2,\ldots,p_{r-1})$.
So there is $y \in A$ with
\[
p_r x \in p_1^m y + (p_2,\ldots,p_{r-1}) \, .
\]
Since
$p_r x \in (p_1^m,p_2,\ldots,p_{r-1}) \subseteq (p_1,\ldots,p_{r-1})$
and $x$ has sufficiently large degree, we have $x \in (p_1,\ldots,p_{r-1})$
by filter-regularity of $p_1,\ldots,p_r$. So there is $z \in A$ with
$x \in p_1 z + (p_2,\ldots,p_{r-1})$.
Hence
$p_1 (p_r z - p_1^{m-1} y) \in (p_2,\ldots,p_{r-1})$.

From Lemma~\ref{lemma:fakAux} we deduce that
$p_r z - p_1^{m-1} y \in (p_2,\ldots,p_{r-1})$ and hence
\[
p_r z \in (p_1^{m-1},p_2,\ldots,p_{r-1}) \, .
\]
As $p_1^{m-1},p_2,\ldots,p_r$ is filter-regular by assumption,
and as~$z$ lies in sufficiently high degree, we deduce
that $z \in (p_1^{m-1},p_2,\ldots,p_{r-1})$. Combining this with
$x \in p_1 z + (p_2,\ldots,p_{r-1})$, we see that
$x \in (p_1^m,p_2,\ldots,p_{r-1})$.
\end{proof}

\begin{lemma}
\label{lemma:fregNilp}
Suppose that $f_1,\ldots,f_r$ and $g_1,\ldots,g_r$ are two sequences
in~$A_+$ with the property that each $g_i - f_i$
is nilpotent. Then $f_1,\ldots,f_r$
is a filter regular sequence if and only if $g_1,\ldots,g_r$ is.
\end{lemma}

\begin{proof}
By nilpotence there is $N_i \geq 1$ such that
$(g_i - f_i)^{N_i} = 0$. Choose $e_i \geq 1$ such that
$p^{e_i} \geq N_i$, then $g_i^{m_i} = f_i^{m_i}$ for $m_i = p^{e_i}$.
Now apply Lemma~\ref{lemma:fregPowers}\@.
\end{proof}

\begin{remark}
Lemmas \ref{lemma:factor}~and \ref{lemma:fregNilp} work well together.
Quotienting out the nilpotent generators reduces computational complexity
and often makes it easier to factorize the parameters, reducing the degrees.
\end{remark}

\begin{bsp}
For group number 38 of order 243,
Lemma~\ref{lemma:filterConstruct} yields parameters in degrees 6, 6,
12 and 16. The parameter in degree 12 has a factor in degree 2, and the last parameter can also be
replaced by one in degree 2. Benson's test then detects
completion in degree~$20$, which is perfect.
\end{bsp}

\section{Existence of filter-regular parameters in low degrees}
\label{sect:existence}
\noindent
We give a non-constructive method that detects
filter-regular parameters in low degrees (Proposition~\ref{Hlem3}), and
adapt Benson's test accordingly (Theorem~\ref{thm:BensonVariant})\@.

\begin{lemma}
\label{Hlem2}
Let $d \geq 1$.  Suppose that $A$~is a finitely generated module over $A(d)$,
the subalgebra generated by~$A_d$.
Then there is a finite field extension $L/k$
and an $x \in L \otimes_k A_d$ such that $\Ann_{A_L}(x)$
is finite-dimensional as a vector space.
\end{lemma}

\begin{proof}
Choose an infinite field~$K$ that is algebraic over~$k$.
Let $x_1,\ldots,x_n$ be a $k$-basis of $A_d$. By assumption,
$A_K = K \otimes_k A$ is finite over $K[x_1,\ldots,x_n]$. Noether
normalisation~\cite[p.~69]{AtiMac} means that there are $K$-linear combinations
$y_1,\ldots,y_r$ of the $x_i$ such that $A_K$ is finite over the
\emph{polynomial} subalgebra $K[y_1,\ldots,y_r]$. Note that each
$y_j \in (A_K)_d$. Let $Y \subseteq (A_K)_d$ be the $K$-vector space
with basis the~$y_i$.

Denote by $A_K^+$ the irrelevant ideal.
As $A_K$ is noetherian we have
$\Ass(A_K)=\{\mathfrak{p}_1,\ldots,\mathfrak{p}_m\}$, a finite set.
Suppose that $Y \subseteq \mathfrak{p}_i$ for some~$i$.
If $x \in A_K^+$ then $x$~is integral over~$K[y_1,\ldots,y_r]$
and therefore $x^s$~lies in the ideal $(Y)$~of $A_K$ for some $s \geq 1$.
So $x \in \mathfrak{p}_i$ and therefore $\mathfrak{p}_i = A_K^+$.

That is, if $\mathfrak{p}_i \neq A_K^+$ then
$\mathfrak{p}_i \cap Y$ is a proper $K$-subspace of~$Y$. As $K$~is infinite,
$Y$ is not a union of finitely many proper subspaces. So there is an
$x \in Y$ which lies in no~$\mathfrak{p}_i$, except possibly
$A_K^+$. Suppose $a \in \Ann_{A_K}(x)$, then $x \in \Ann_{A_K}(a)$,
so $x \in \Ann_{A_K}(ab) \in \Ass A_K$ for some $b \in A_K$.
Hence $\Ann_{A_K}(ab)=A_K^+$, and therefore
$ab \in \Ann_{A_K}(A_K^+)$. But $\Ann_{A_K}(A_K^+)$ is finite
dimensional, as $A_K$ is noetherian and connected.
So there is a $N \geq 1$ such that $\left|c\right| \leq N$ for every
$c \in \Ann_{A_K}(A_K^+)$. But then $\left|a\right| \leq N$,
so $\Ann_{A_K}(x)$ is finite dimensional too.

Finally write $x$ as a $K$-linear combination of the $k$-basis $x_1,\ldots,x_n$
of~$A_d$. Let $L \supseteq k$ be the field generated
by the coefficients. Then $L/k$ is finite.
\end{proof}

\begin{proposition}
\label{Hlem3}
Let $d \geq 1$.  Suppose that $A$~is a finitely generated module over $A(d)$,
the subalgebra generated by~$A_d$.
Then there is a finite extension $K/k$
and a filter-regular system of parameters $x_1,\ldots,x_n$ for $A_K$ with
$\left|x_i\right| = d$ for every~$i$.
\end{proposition}

\begin{proof}
If $\dim(A)=0$ then
$\dim_k(A)< \infty$ and we are done, with $n=0$. If $\dim(A) \geq 1$
then by Lemma~\ref{Hlem2} there is a finite extension $L/k$ and
an $x_1 \in (A_L)_d$ such that $\Ann_{A_L}(x_1)$ is finite dimensional.
This means that $\dim(A_L/(x_1)) < \dim(A)$. The result follows by induction over
the Krull dimension.
\end{proof}

\subsection*{A variant of Benson's test for completion}
Benson's test for completion
\cite[Theorem~10.1]{Benson:DicksonCompCoho} makes two separate uses of a filter-regular
system of parameters for $H^*(G,k)$. First one determines the filter degree type.
Then one obtains a degree bound by combining the type with the sum of the parameter degrees.

Determining the type calls for explicitly constructed parameters,
which are often in high degree.
We now show that one can use a different parameter system at each of the two stages.
This allows us to detect completion earlier using Proposition~\ref{Hlem3},
as an existence proof is
all that is needed for the degree sum.
\medskip

\noindent
Following Benson \cite[\S10]{Benson:DicksonCompCoho} we write
$\tau_N H^*(G,k)$ for the approximation to the graded commutative
ring $H^*(G,k)$ obtained by taking all generators and relations in
degree${}\leq N$.
Recall that if $K/k$~is a field extension then
$H^*(G,K) \cong K \otimes_k H^*(G,k)$, and hence
$\tau_N H^*(G,K) \cong K \otimes_k \tau_N H^*(G,k)$.
The following result is based on Theorem~10.1 of
Benson's paper~\cite{Benson:DicksonCompCoho}\@.
% Recall that the $p$-rank of $G$ is the rank of the largest elementary
% abelian $p$-subgroup.

\begin{theorem}
\label{thm:BensonVariant}
Let $G$ be a finite group, $k$ a field of characteristic $p$, and $K/k$ a field
extension.
Suppose that $r = \prank(G) \geq 2$. Suppose that
$\zeta_1,\ldots,\zeta_r \in \tau_N H^*(G,k)$ and
$\kappa_1,\ldots,\kappa_r \in \tau_N H^*(G,K)$
have the following properties:
\begin{enumerate}
\item
Each system is a filter-regular homogeneous system of parameters in its
respective ring.
Let $(d_0,\ldots,d_r)$ be the type of $\zeta_1,\ldots,\zeta_r$.
\item
The images of the~$\zeta_i$ form a homogeneous system of
parameters in $H^*(G,k)$.
\item
Set $n_i = \abs{\kappa_i}$ and
$\alpha' = \max_{0 \leq i \leq r-2} (d_i + i)$.
Then $n_i \geq 2$ for all~$i$, and
\begin{equation}
\label{eqn:BensonVariant}
N > \max(\alpha',0) + \sum_{j=1}^r (n_j - 1) \, .
\end{equation}
\end{enumerate}
Then the map $\tau_N H^*(G,k) \rightarrow H^*(G,k)$ is an isomorphism.

Moreover, if a Sylow $p$-subgroup of~$G$ has centre of $p$-rank at least two,
then in Eqn.~\eqref{eqn:BensonVariant}
we only have to require~$\geq$\@.
\end{theorem}

\begin{proof}
Since $H^*(G,K) \cong K \otimes_k H^*(G,k)$, the~$\zeta_i$
are a filter-regular system of parameters of type
$(d_0,\ldots,d_r)$ in $\tau_N H^*(G,K)$ as well.
By (i)${}\Rightarrow{}$(iii) in
Theorem~4.5 of~\cite{Benson:DicksonCompCoho}, it follows that
$a^i_{\mathfrak{m}}(\tau_N H^*(G,K)) \leq d_i$ for each $0 \leq i \leq r$.
Hence $\alpha' \geq \alpha$ for $\alpha$ as in Benson's Theorem 10.1\@.
Since the $\zeta_i$ are a system of parameters for $H^*(G,K)$,
so are the~$\kappa_i$. Hence applying Benson's Theorem~10.1 to the $\kappa_i$
we deduce that
$\tau_N H^*(G,K) \cong H^*(G,K)$ and therefore
$\tau_N H^*(G,k) \cong H^*(G,k)$. The last part follows
from the corrected version\footnote{It is the \emph{centre}
of the Sylow $p$-subgroup which should have rank at least two.}
of Benson's Remark 10.6(v)\@.
\end{proof}

\begin{bsp}
For the Sylow $2$-subgroup of~$\mathit{Co}_3$,
the computer finds filter-regular parameters in degrees $8$, $4$, $6$,~$1$\@.
Killing the first two
parameters and then applying Proposition~\ref{Hlem3} shows
that there are filter-regular parameters in degrees $8$, $4$, $2$,~$2$.
So Theorem~\ref{thm:BensonVariant} allows us to apply Benson's test in degree~$14$
(which is perfect).
\end{bsp}

\section{Notes on the Implementation}
\label{sect:notes}
\noindent
Our strategy for computing the cohomology of
all groups of order 128 was to combine improved theoretical methods
(Proposition~\ref{Hlem3},
Lemma~\ref{lemma:filterConstruct} and Lemma~\ref{lemma:factor})
with an efficient implementation.

The first author's package~\cite{habil} was not powerful enough,
nor was it in a form suitable for distribution.
The C routines for minimal projective resolutions~\cite{habil}
were however serviceable enough,
and reimplementing them from scratch would have
been costly.  We still needed \textsf{GAP}~\cite{GAP4} for the group-theoretic
work, and in addition we now needed a commutative
algebra system such as \textsc{Singular}~\cite{Singular}.

The Sage project~\cite{Sage} perfectly matched our
needs. It provides a
common interface for and contains a distribution of several independent
computer algebra systems, including
\textsc{Singular}~\cite{Singular} and \textsf{GAP}.
Sage is based on Cython~\cite{Cython}, which is a
compiled version of Python. It makes linking against C-code easy,
and it yields a huge speed-up compared with interpreted code.
Sage is free open source software, so that it is more easy for us to make our
results publicly available. 

The second author implemented the (Yoneda) product and several other homological algebra constructions
in Cython~\cite{Cython}.
He also implemented
the computation of the ring structure and of the completeness criterion provided by
Theorem~\ref{thm:BensonVariant}. For this purpose, fast Gr{\"o}b\-ner basis computations in
graded commutative rings are particularly important. This is provided by 
\textsc{Singular}~\cite{Singular}.

Our project resulted in an optional Sage package,
called the $p$-Group Cohomology Package~\cite{SimonsProg}\@, with constituents as described above.
%Its constituent components are as follows:
%\begin{itemize} 
% \item A modified version of the Aachen C-MeatAxe~\cite{MeatAxe224} for
% linear algebra over finite fields
%
% \item The first author's C-routines
% for the computation of minimal resolutions
%
% \item \textsf{GAP} functions for computing elementary abelian
% subgroups, dealing with groups from the SmallGroups
% library~\cite{BeEiOBr:Millennium} and the Atlas of finite groups~\cite{ATLAS}
%
% \item \textsc{Singular} functions that compute a Gr\"obner basis of
% the relation ideal, detect new relations, detect filter-regular
% systems of parameters and compute their filter degree type, and
% (optionally) construct simultaneous lifts of the Dickson elements
%
% \item \textsc{Cython} modules written by the second author, providing
% methods for homological algebra computations (cup product,
% restriction to maximal elementary abelian subgroups, induced
% homomorphisms, etc)
%\end{itemize}
%
The package also provides general Massey products and restricted
Massey powers. The latest version features Persistent Group
Cohomology, introduced by Graham Ellis and the second
author~\cite{EllisKing:Persistent}.

\section{Experimental Results}
\label{sect:Experimental}
\noindent
We computed the mod-$2$ cohomology rings of all 2328 groups of order 128\@.
The results -- omitted here for space reasons -- may be
consulted online~\cite{GreenKing:128website} and read into our cohomology package~\cite{SimonsProg}.
Further cohomology computations include:
\begin{itemize}
\item Some further $2$-groups,
including the Sylow 2-subgroups of $\mathit{HS}$~and $\mathit{Co}_3$;
\item All but six groups of order dividing $3^5$, together with some
of order $3^6$; and
\item All groups of order dividing $5^4$.
%
% \item all 7-groups up to order 343.
\end{itemize}
Each computation includes a minimal ring presentation, the Poincar{\'e} series and the depth.
Our results agree with Carlson's for the groups of
order~$64$~\cite[Appendix]{CarlsonTownsley}, and
with the first author's for groups of order $3^4$~and $5^4$~\cite{habil}\@.
% We can also confirm the computation of Adem,
% Carlson, Karagueuzian and Milgram~\cite{ACKM} of the cohomology of the
% Sylow 2-subgroup of the Higman-Sims group.

The first complete computation for order~$128$ was in Summer 2008,
using Lemmas \ref{lemma:factor}~and \ref{lemma:filterConstruct} to improve Benson's criterion.
Split over several processors, the real time used adds up to roughly 10 months.
Adding in the existence result for low-degree parameters
(Proposition~\ref{Hlem3}), reduced the
computation time to a total of about 2 months, on a Sun X4450
machine with 2.66~GHz.

Our package can now compute the cohomology of
all groups of order 64 in about 20 minutes on a Mac Pro
desktop machine running at 2.6 Ghz with 8 GB RAM (Darwin Kernel
Version 10.0.0) in 64 bit mode. 

% Another new result is the cohomology ring of the Sylow 2-subgroup of
% the third Conway group. Actually we started with two different group
% presentations and obtained consistent results. The ring is of depth 3. 
% Its Poincar\'e series is
% \[
% \frac{t^{9} + t^{8} + 2 t^{7} + 3 t^{6} + t^{5} + 3 t^{4} + 2 t^{3} + 3 t^{2} + t + 1}{t^{13} - 3 t^{12} + 4 t^{11} - 4 t^{10} + 3 t^{9} - t^{8} - t^{5} + 3 t^{4} - 4 t^{3} + 4 t^{2} - 3 t + 1}
% \]
% The $a$-invariants are $-\infty, -\infty, -\infty, -6, -4$.
% 
% A minimal algebra presentation of the cohomology ring would fill
% about 20 pages, thus we don't include it here. Two minimal algebra
% presentations can be found on our web pages~\cite{GreenKing:128website}.
% \bigskip

\subsection*{Extremal cases}
\noindent
We highlight some order~$128$ groups whose cohomology is extremal in various ways.
We adopt the numbering of the Small Groups library~\cite{BeEiOBr:Millennium}.

The improved Benson criterion (Theorem~\ref{thm:BensonVariant}) is very efficient.
In 1779 out of 2328 cases it already applies in the degree of
the last generator or relation.
The worst performance is for $2^4 \times D_8$ (number 2320): the presentation is complete from
degree~$2$ onwards, but this is only detected $4$ degrees later.

Number 836 is the Sylow $2$-subgroup of one double cover of $\mathit{Sz}(8)$.
Its cohomology ring has 65 minimal generators and 1859 minimal relations. This is the largest presentation by
far: the other groups have at most 39 generators and at most 626 relations. But it is group number
562 which has the highest generator degree (17) and the highest relation degree (34)\@.

The longest computation took about 11 days. Surprisingly, the group concerned (number 2327) was~$2^{1+6}_-$\@:
the cohomology of this extraspecial group has been known since
1971~\cite{Quillen:Extraspecial} and is easy to write down. The computations
for groups number 2298 and 2300 took about a week each.

\begin{rk}
Other groups of order $128$ have been studied
before too.
Adem--Milgram \cite{AdMi:M22} determined the cohomology of group $931$,
and Maginnis~\cite{Maginnis:J2} treated group $934$:
these are the Sylow subgroups of the sporadic groups
$M_{22}$ and $J_2$.
\end{rk}

\subsection*{$a$-invariants}
Benson--Carlson duality and Symonds' regularity theorem
are two instances of the significance of the
local cohomology of $H^*(G)$~\cite{Benson:MSRI}\@.
Regularity measures the offset of the vanishing line of (bigraded)
local cohomology. The $a$-invariants give more detailed vanishing information.

Let $k$~be a field and $R$ a connected finitely presented graded commutative
$k$-algebra. Let $M$ be a finitely generated graded $R$-module, and
$\mathfrak{m}$ the ideal in~$R$ of all positive degree elements.
Recall that the $a$-invariants of $M$ are defined by
\[
a_{\mathfrak{m}}^i(M) = \max \{m \mid H^{i,m}_{\mathfrak{m}}(M) \neq 0 \} \, ,
\]
with $a_{\mathfrak{m}}^i(M) = -\infty$ if $H^i_{\mathfrak{m}}(M)=0$.
Benson provides a recipe for computing the $a$-invariants from a
filter-regular system of parameters \cite[Lemma~4.3]{Benson:DicksonCompCoho}.

We computed these invariants for all groups of
orders 64~and 128\@.
Interestingly, amongst the 2595 groups of these
orders there are only 15 cases where the $a$-invariants are not weakly
increasing. These are listed in Table~\ref{table:nonMonotone}\@.

\begin{table}
\renewcommand{\arraystretch}{1.5}
$\begin{array}{|c|c|cccccc|c|}
\hline
\abs{G} & \text{Group}
& a^0_{\mathfrak{m}}
& a^1_{\mathfrak{m}}
& a^2_{\mathfrak{m}}
& a^3_{\mathfrak{m}}
& a^4_{\mathfrak{m}}
& a^5_{\mathfrak{m}} & \text{Note} \\
\hline \hline
64 & 0242 & -\infty & -\infty & -3 & -5 & -4 & &
\text{Sylow in $L_3(4)$} \\
\hline \hline
\multirow{14}{*}{$128$} &
0391 & -\infty & -\infty & -3 & -5 & -4 & & \\
\cline{2-9}
& 0741 & -\infty & -4      & -5 & -3 &    & & \\
\cline{2-9}
& 0749 & -\infty & -3      & -7 & -3 &    & & \\
\cline{2-9}
& 0836 & -\infty & -4      & -5 & -4 & -4 & &
\text{Sylow in one $2.\mathit{Sz}(8)$} \\
\cline{2-9}
& 0931 & -\infty & -\infty & -3 & -4 & -4 & & \text{Sylow in $M_{22}$} \\
\cline{2-9}
& 0934 & -\infty & -\infty & -3 & -5 & -4 & & \text{Sylow in $J_2$} \\
\cline{2-9}
& 1411 & -\infty & -\infty & -\infty & -4 & -6 & -5 & \\
\cline{2-9}
& 1931 & -\infty & -\infty & -3 & -5 & -4 & & \\
\cline{2-9}
& 2005 & -\infty & -\infty & -3 & -5 & -4 & & \\
\cline{2-9}
& 2191 & -\infty & -\infty & -4 & -6 & -4 & & \\
\cline{2-9}
& 2258 & -\infty & -\infty & -\infty & -4 & -6 & -5 & \text{Sylow in
$2 \times L_3(4)$} \\
\cline{2-9}
& 2261 & -\infty & -\infty & -3 & -5 & -4 & & \\
\cline{2-9}
& 2272 & -\infty & -\infty & -3 & -5 & -4 & & \\
\cline{2-9}
& 2300 & -\infty & -\infty & -3 & -4 & -4 & & \\
\hline
\end{array}$
\vspace*{10pt}
\caption{The fifteen groups of orders 64 and 128 whose $a$-invariants
are not weakly increas�ng.}
\label{table:nonMonotone}
\end{table}

\section{Depth and the Duflot bound}
\label{sect:Duflot}
\noindent
Several results relate the group structure
of a finite group~$G$ to the commutative algebra of its mod-$p$ cohomology ring
$H^*(G)$.
For the Krull dimension and depth we have the following inequalities,
where $S$ denotes a Sylow $p$-subgroup of~$G$:
\begin{equation}
\label{eqn:QuillenDuflot}
\prank(Z(S)) \leq \depth H^*(S) \leq \depth H^*(G)
\leq \dim H^*(G) = \prank(G) \, .
\end{equation}
The first inequality is Duflot's theorem~\cite[\S12.3]{CarlsonTownsley}\@.
The second is Benson's \cite[Thm~2.1]{Benson:NYJM2}
and
``must be well known''\@. The third is
automatic for finitely generated connected $k$-algebras.
The last equality is due to Quillen~\cite[Coroll~8.4.7]{CarlsonTownsley}\@.
% These inequalities motivate the following definitions.

\begin{defn}
Let $G, p, S$ be as above. The
Cohen--Macaulay defect $\CMd_p(G)$
and the
Duflot excess $\De_p(G)$
are defined by
\begin{xalignat*}{2}
\CMd(G) & = \dim H^*(G) - \depth H^*(G) &
\De(G) & = \depth H^*(G) - \prank(Z(S)) \, .
\end{xalignat*}
\end{defn}

\noindent
It follows from Eqn.~\eqref{eqn:QuillenDuflot} that
\begin{xalignat}{3}
\label{eqn:gtD-CMd}
\CMd(G), \De(G) & \geq 0 &
\CMd(G) & \leq \CMd(S) &
\De(G) & \geq \De(S) \, .
\end{xalignat}
The term Cohen--Macaulay defect
(sometimes deficiency)
is already in use.

\begin{ques}
How (for large values of~$n$) are the $p$-groups of order $p^n$ distributed
on the graph with $\CMd(G)$ on the $x$-axis and $\De(G)$ on the $y$-axis?
\end{ques}

\begin{bspe}
Every abelian group has $(\CMd, \De) = (0,0)$. So does every $p$-central group, such as
the Sylow $2$-subgroup of $\mathit{SU}_3(4)$.

Quillen showed that the extraspecial $2$-group $2^{1+2n}_+$ has
Cohen--Macaulay cohomology~\cite{Quillen:Extraspecial}.
So this group has $(\CMd,\De) = (0,n)$.
For odd~$p$, Minh proved~\cite{Minh:EssExtra} that $p^{1+2n}_+$
has $(\CMd, \De) = (n,0)$. With one exception:
$3^{1+2}_+$ has $(\CMd,\De) = (0,1)$~\cite{MilgramTezuka}\@.

One way to produce groups with small $\De(G)/\CMd(G)$ ratio
is by iterating the wreath product construction. By passing
from $H$~to $H \wr C_p$ one multiplies the $p$-rank by~$p$ but increases the
depth by one only~\cite{CaHe:Wreath}\@.
\end{bspe}

\noindent
Table~\ref{table:de128} lists the number of groups with each
value of $(\CMd,\De)$, for various orders.
% For $\abs{G} = 64$ this information can be
% read off from the tables in~\cite[Appendix]{Benson:MSRI}\@.

\begin{table}
$\begin{tabular}{c@{\hspace*{25pt}}c|c@{\hspace*{25pt}}c}
$\begin{array}{cc|c|c|c|c|}
\cline{3-6}
\multirow{4}{*}{$\De$} & 3 & 0 & & & \\
\cline{3-6}
& 2 & 5 & 1 & & \\
\cline{3-6}
& 1 & 45 & 22 & 1 & \\
\cline{3-6}
& 0 & 69 & 103 & 21 & 0 \\
\cline{3-6}
& \multicolumn{1}{c}{}
& \multicolumn{1}{c}{0}
& \multicolumn{1}{c}{1}
& \multicolumn{1}{c}{2}
& \multicolumn{1}{c}{3} \\
\multicolumn{2}{c}{N=64} & \multicolumn{4}{c}{\CMd}
\end{array}$
& & &
$\begin{array}{cc|c|c|c|c|}
\cline{3-6}
\multirow{4}{*}{$\De$} & 3 & 1 & & & \\
\cline{3-6}
& 2 & 27 & 10 & & \\
\cline{3-6}
& 1 & 256 & 220 & 32 & \\
\cline{3-6}
& 0 & 353 & 1123 & 292 & 14 \\
\cline{3-6}
& \multicolumn{1}{c}{}
& \multicolumn{1}{c}{0}
& \multicolumn{1}{c}{1}
& \multicolumn{1}{c}{2}
& \multicolumn{1}{c}{3} \\
\multicolumn{2}{c}{N=128} & \multicolumn{4}{c}{\CMd}
\end{array}$
\\ &&& \\ \hline &&& \\
$\begin{array}{cc|c|c|c|}
\cline{3-5}
\multirow{3}{*}{$\De$} & 2 & 0 & & \\
\cline{3-5}
& 1 & 2 & 1 & \\
\cline{3-5}
& 0 & 6 & 6 & 0 \\
\cline{3-5}
& \multicolumn{1}{c}{}
& \multicolumn{1}{c}{0}
& \multicolumn{1}{c}{1}
& \multicolumn{1}{c}{2} \\
\multicolumn{2}{c}{N=3^4} & \multicolumn{3}{c}{\CMd}
\end{array}$
& & &
$\begin{array}{cc|c|c|c|}
\cline{3-5}
\multirow{3}{*}{$\De$} & 2 & 0 & & \\
\cline{3-5}
& 1 & 0 & 0 & \\
\cline{3-5}
& 0 & 6 & 7 & 2 \\
\cline{3-5}
& \multicolumn{1}{c}{}
& \multicolumn{1}{c}{0}
& \multicolumn{1}{c}{1}
& \multicolumn{1}{c}{2} \\
\multicolumn{2}{c}{N=5^4} & \multicolumn{3}{c}{\CMd}
\end{array}$
\end{tabular}$
\vspace*{10pt}
\caption{Distribution of the groups of order $N$ by defect~$\CMd$ and excess~$\De$}
\label{table:de128}
\end{table}

\begin{remarks}
\label{rks:graph}
The group of order~64 with $(\CMd,\De)=(1,2)$ is group number 138\@.
The one with $(\CMd,\De)=(2,1)$ is number 32\@.

The 14 groups of order 128 with $(\CMd,\De) = (3,0)$ are
investigated in Proposition~\ref{prop:calc} below.
The one with $(\CMd,\De) = (0,3)$ is the extraspecial
group $2^{1+6}_+$\@.

The group of order 81 with $(\CMd,\De) = (1,1)$ is the Sylow
3-subgroup of~$S_9$. One of the groups of order 625 with $(\CMd,\De) = (2,0)$
is the Sylow 5-subgroup of the Conway group $\mathit{Co}_1$.
\end{remarks}

\begin{ques}
\label{ques:Duflot}
Is the Duflot bound sharp for most groups of a given order~$p^n$? That is,
do most groups of order $p^n$ lie on the bottom row $\De=0$ of the
$\CMd/\De$-graph? If so, is this effect more pronounced for large primes~$p$?
\end{ques}

\section{The strong form of the regularity conjecture}
\label{sect:strong}
\noindent
Benson conjectured~\cite{Benson:DicksonCompCoho} and Symonds proved~\cite{Symonds:RegProof} that
$H^*(G,k)$ has Castelnuovo--Mumford
regularity zero. Equivalently, the cohomology ring $H^*(G,k)$ always contains a strongly
quasi-regular system of parameters.
Kuhn has applied Symonds' result to
the study of central essential cohomology~\cite{Kuhn:Cess}\@.
The following stronger form of the conjecture remains open.
% It originated as a remark on p.~175 of Benson's paper.

\begin{conj}[Benson: see p.~175 of~\cite{Benson:DicksonCompCoho}]
\label{conj:VSQR}
Let $G$ be a finite group, $p$ a prime number and $k$ a field
of characteristic~$p$. The cohomology ring $H^*(G,k)$ always contains a
very strongly quasi-regular system of parameters.
\end{conj}

\noindent
We are going to verify the conjecture for all groups of order${} < 256$.
We use the Cohen--Macaulay defect $\CMd(G)$ to reduce to the case $\abs{G}
= 128$, and verify it there by machine computation.

\begin{theorem}[Benson]
\label{thm:portfolio}
% Let $G$ be a finite group, $p$ a prime number and $k$ a field
% of characteristic~$p$.
If $\CMd(G) \leq 2$ then
Conjecture~\ref{conj:VSQR} holds for~$G$.
This is the case for every group of order~$64$.
% In particular, every group of order~$64$ satisfies the conjecture.
\end{theorem}

\begin{proof}
This is Theorem~1.5 of~\cite{Benson:DicksonCompCoho}\@. 
Carlson~\cite{Carlson:Online3,CarlsonTownsley} observed that
every group of order $64$ has $\CMd(G) \leq 2$
See also the tabular data in~\cite[Appendix]{Benson:MSRI}\@.
\end{proof}

\noindent
So we are only interested in groups with $\CMd(G) \geq 3$. Now,
$\CMd(G)$ has a group-theoretic upper bound.
Let $S \leq G$ be a Sylow $p$-subgroup. Then by Eqn~\eqref{eqn:QuillenDuflot}
\[
\CMd(G) \leq \gtD(G) := \prank(G) - \prank(Z(S)) \, .
\]

\begin{lemma}
\label{lemma:Jordan}
Let $G$ be a finite group and $p$ a prime.
\begin{enumerate}
\item
\label{enumi:Jordan3}
If $\gtD(G) \geq 3$ then $p^5 \mid \abs{G}$, and
if $p=2$ or $p=3$ then $p^6 \mid \abs{G}$.
\item
\label{enumi:Jordan4}
If $\gtD(G) \geq 4$ then $p^6 \mid \abs{G}$, and if
$p=2$ or $p=3$ then $p^7 \mid \abs{G}$.
\item
\label{enumi:Jordan128}
If $p=2$ and $\gtD(G) \geq 4$ then $2^8 \mid \abs{G}$.
\end{enumerate}
\end{lemma}

\begin{proof}
Clearly $\gtD(G) = \gtD(S)$, so we may assume that
$G$~is a $p$-group.
\medskip

\noindent
\ref{enumi:Jordan3}):
If $\gtD(G) \geq 3$ then $G \neq 1$ and so $\prank(Z(G)) \geq 1$, whence
$\prank(G) \geq 4$. Moreover, $G$ must be non-abelian, so
$\abs{G} \geq p^5$.

Suppose that $\abs{G} = p^5$.
Then $Z(G)$ is cyclic of order~$p$,
and there is an elementary abelian subgroup~$V$ of order~$p^4$.
By maximality, $V \trianglelefteq G$.
Let $a \in G \setminus V$. Then $G = \langle a,V\rangle$
and the conjugation action of $a$~on $V$ is nontrivial of order~$p$.
So the minimal polynomial of the action divides $X^p-1 = (X-1)^p$.
Consider the Jordan normal form of this action. Each block contains
an eigenvector, which must lie in $Z(G)$. So as $Z(G)$ is cyclic,
there is just one Jordan block, of size~$4$\@.
But for $p=2$
the size~$3$ Jordan block does not square to the identity, so each block
must have size${} \leq 2$.
Similarly there can be no size~$4$ block
for $p=3$, since it does not cube to the identity.
The proof of~\ref{enumi:Jordan4}) is analogous.
\medskip

\noindent
\ref{enumi:Jordan128}):
We assume that $\abs{G} = 128$ and derive a contradiction.
Let $V \leq G$ be elementary abelian of rank $r = \prank(G)$, then
$r \geq 4 + \prank(Z(G))$. Since $G$~is non-abelian, $r = 5$ or $6$.
If $r = 6$ then we are in a similar situation
to~\ref{enumi:Jordan3}): the Jordan blocks must have size${} < 3$, so we
need at least three; but there can be at most two, since $\prank(Z(G)) \leq 2$.

So $r = 5$ and $\prank(Z(G)) = 1$.
Pick $V \lneq H \lneq G$, so $[G:H]=[H:V]=2$.
As $V$~is maximal elementary abelian in $H$, we have $C \leq V$ for
$C = \Omega_1(Z(H))$. The Jordan block argument means that
$C$~has rank${}\geq 3$. Applying this argument to the action
of $G/H$ on~$C$ then yields $\prank(Z(G)) \geq 2$, a contradiction.
\end{proof}

\begin{remark}
For $p \geq 5$ consider the action of a size 4 Jordan block on a rank four elementary abelian. This yields
a group of order $p^5$ with $\gtD = 3$, since
\[
\begin{pmatrix}
1 & 1 & 0 & 0 \\ 0 & 1 & 1 & 0 \\ 0 & 0 & 1 & 1 \\ 0 & 0 & 0 & 1 \end{pmatrix}^n
= \begin{pmatrix} 1 & \binom n1 & \binom n2 & \binom n3 \\
0 & 1 & \binom n1 & \binom n2 \\ 0 & 0 & 1 & \binom n1 \\ 0 & 0 & 0 & 1
\end{pmatrix} \, .
\]
\end{remark}

\begin{corollary}
\label{coroll:only128}
If $\abs{G} < 256$ then $\CMd(G) \leq 3$, with equality only if
$\abs{G} = 128$.
\end{corollary}

\begin{proof}
Since $\CMd(G) \leq \gtD(G)$ and $\abs{G} < 256$,
Lemma~\ref{lemma:Jordan} says that $p = 2$ and $64 \mid \abs{G}$.
So since $\CMd(G) \leq \CMd(S)$ and $\abs{G} < 256$,
it suffices to exclude the case $\abs{G} = 64$. But this follows from
Carlson's computations (see Theorem~\ref{thm:portfolio})\@.
\end{proof}

\noindent
Lemma~\ref{lemma:Jordan}\,\ref{enumi:Jordan128})
tells us that if $\abs{G} = 128$ and $\CMd(G) = 3$
then $\gtD(G) = 3$.

\begin{proposition}
\label{prop:calc}
Only $57$ groups of order $128$ satisfy $\gtD(G) = 3$, and of these
only $14$ satisfy $\CMd(G)=3$. Each of these $14$ groups
satisfies Conjecture~\ref{conj:VSQR}\@.

In the Small Groups
Library~\cite{BeEiOBr:Millennium}, these $14$ groups have the
numbers
$36$, $48$, $52$, $194$, $515$, $551$, $560$, $561$, $761$, $780$, $801$,
$813$, $823$ and $836$.
\end{proposition}

\begin{table}
\settowidth{\djglength}{$0000$}
\newcommand{\gr}[1]{\makebox[\djglength][r]{$#1$}}
\newcommand{\ugr}[1]{\makebox[\djglength][r]{\underline{$#1$}}}
$\begin{array}{|c|cccc|c|c|cccc|}
\cline{1-5} \cline{7-11}
\text{gp} & K & d & r & \CMd & \quad & \text{gp} & K & d & r & \CMd \\
\cline{1-5} \cline{7-11}
\ugr{36} & 5 & 2 & 2 & 3 & & \gr{850} & 5 & 3 & 2 & 2 \\
\ugr{48} & 5 & 2 & 2 & 3 & & \gr{852} & 4 & 2 & 1 & 2 \\
\ugr{52} & 4 & 1 & 1 & 3 & & \gr{853} & 4 & 2 & 1 & 2 \\
\ugr{194} & 5 & 2 & 2 & 3 & & \gr{854} & 4 & 2 & 1 & 2 \\
\gr{513} & 5 & 3 & 2 & 2 & & \gr{859} & 4 & 2 & 1 & 2 \\
\ugr{515} & 5 & 2 & 2 & 3 & & \gr{860} & 4 & 2 & 1 & 2 \\
\gr{527} & 4 & 2 & 1 & 2 & & \gr{866} & 4 & 2 & 1 & 2 \\
\ugr{551} & 5 & 2 & 2 & 3 & & \gr{928} & 4 & 3 & 1 & 1 \\
\ugr{560} & 4 & 1 & 1 & 3 & & \gr{929} & 4 & 2 & 1 & 2 \\
\ugr{561} & 4 & 1 & 1 & 3 & & \gr{931} & 4 & 2 & 1 & 2 \\
\gr{621} & 5 & 3 & 2 & 2 & & \gr{932} & 4 & 2 & 1 & 2 \\
\gr{623} & 4 & 2 & 1 & 2 & & \gr{934} & 4 & 2 & 1 & 2 \\
\gr{630} & 5 & 3 & 2 & 2 & & \gr{1578} & 6 & 4 & 3 & 2 \\
\gr{635} & 4 & 2 & 1 & 2 & & \gr{1615} & 4 & 3 & 1 & 1 \\
\gr{636} & 4 & 2 & 1 & 2 & & \gr{1620} & 4 & 2 & 1 & 2 \\
\gr{642} & 4 & 2 & 1 & 2 & & \gr{1735} & 5 & 3 & 2 & 2 \\
\gr{643} & 4 & 2 & 1 & 2 & & \gr{1751} & 4 & 2 & 1 & 2 \\
\gr{645} & 4 & 3 & 1 & 1 & & \gr{1753} & 4 & 3 & 1 & 1 \\
\gr{646} & 4 & 2 & 1 & 2 & & \gr{1755} & 5 & 4 & 2 & 1 \\
\gr{740} & 4 & 2 & 1 & 2 & & \gr{1757} & 4 & 3 & 1 & 1 \\
\gr{742} & 4 & 2 & 1 & 2 & & \gr{1758} & 4 & 3 & 1 & 1 \\
\gr{753} & 5 & 3 & 2 & 2 & & \gr{1759} & 4 & 3 & 1 & 1 \\
\ugr{761} & 5 & 2 & 2 & 3 & & \gr{1800} & 4 & 2 & 1 & 2 \\
\gr{764} & 4 & 2 & 1 & 2 & & \gr{2216} & 5 & 4 & 2 & 1 \\
\ugr{780} & 4 & 1 & 1 & 3 & & \gr{2222} & 5 & 3 & 2 & 2 \\
\ugr{801} & 4 & 1 & 1 & 3 & & \gr{2264} & 5 & 3 & 2 & 2 \\
\ugr{813} & 4 & 1 & 1 & 3 & & \gr{2317} & 4 & 3 & 1 & 1 \\
\ugr{823} & 4 & 1 & 1 & 3 & & \gr{2326} & 4 & 4 & 1 & 0 \\
\cline{7-11}
\ugr{836} & 4 & 1 & 1 & 3 \\
\cline{1-5}
\end{array}$
\vspace*{5pt}
\caption{For each group of order $128$ with $\gtD(G)=3$, we give its number
in the Small Groups library, the Krull dimension $K$ and
depth $d$ of $H^*(G)$, the rank~$r = K - 3$ of $Z(G)$ and the
Cohen--Macaulay defect $\CMd = K - d$. Underlined entries have $\CMd = 3$\@.
% Notation based on that of~\cite[Appendix]{Benson:MSRI}\@.
}
\label{table:128}
\end{table}

\begin{proof}
The invariant $\gtD$ is purely group theoretic. Determining it for each of the
$2328$ groups of order $128$ yields only $57$ cases
with $\gtD(G) = 3$. The rest follows by inspecting our cohomology computations:
as discussed above, determining the filter degree type is
an integral part of our computations.
Table~\ref{table:128} lists each of these groups together with its defect.
\end{proof}

\begin{theorem}
\label{thm:main}
Conjecture~\ref{conj:VSQR} holds for every group of order less than 256\@.
\end{theorem}

\begin{proof}
Follows from Theorem~\ref{thm:portfolio},
Corollary~\ref{coroll:only128} and
Proposition~\ref{prop:calc}\@.
\end{proof}

\begin{remark}
Testing Conjecture~\ref{conj:VSQR} further requires more high defect groups.
Carlson~\cite{Carlson:DepthTransfer} showed that
if there are essential classes in $H^*(G)$, then
equality holds in $\CMd(G) \leq \gtD(G)$\@.
This method shows that groups number $35$, $56$ and $67$ of order~$3^6$
have $\CMd(G) = 3$; and group $299$ of order $256$ has
$\CMd(G) = 4$. In each case there is an essential class in degree${} \leq 4$.
\end{remark}

\section*{Acknowledgements}
\noindent
The idea for Theorem~\ref{thm:BensonVariant} arose during a
discussion with Dave Benson.
We thank the referee for advice on the structure of the paper.

This work was supported by the German Science
Foundation (DFG), project numbers GR 1585/4-1 and -2\@.

We are grateful to William Stein for giving us the opportunity to work on
the \texttt{sage.math} computer, which is supported by National
Science Foundation Grant No. DMS-0821725.

% \bibliographystyle{abbrv}
% \bibliography{united}

\end{document}